\documentclass[12pt]{article}
\usepackage[centertags]{amsmath}
\usepackage{amsfonts}
\usepackage{amssymb}
\usepackage{amsthm}
\input{amssym.def}
\usepackage{graphicx}

\numberwithin{equation}{section}

\newtheorem{Definition}{Definition}[section]
\newtheorem{theorem}[Definition]{Theorem}
\newtheorem{lemma}[Definition]{Lemma}
\newtheorem{proposition}[Definition]{Proposition}
\newtheorem{corollary}[Definition]{Corollary}

\newtheorem{result}[Definition]{Result}

\begin{document}
\title{\Large \bf On some characterizations of  strong power graphs of finite groups}
\author{A. K. Bhuniya and Sudip Bera}
\date{}
\maketitle

\begin{center}
Department of Mathematics, Visva-Bharati, Santiniketan-731235, India. \\
anjankbhuniya@gmail.com, sudipbera517@gmail.com
\end{center}

\begin{abstract}
Let $ G $ be a finite group  of order $ n$. The strong power graph $\mathcal{P}_s(G) $ of $G$ is the
undirected graph whose vertices are the elements of $G$ such that two distinct vertices $a$ and $b$
are adjacent if  $a^{{m}_1}$=$b^{{m}_2}$ for some positive integers ${m}_1 ,{m}_2 < n$. In this article we classify all groups $G$ for which $\mathcal{P}_s(G)$ is line graph and Caley graph. Spectrum and permanent of the Laplacian matrix of the strong power graph $\mathcal{P}_s(G)$ are found for any finite group $G$.
\end{abstract}

\section{Introduction}
The study of different algebraic structures using graph theory becomes an exciting research topic in the last few decades, leading to many fascinating results and questions. Given an algebraic structure $S$, there are different formulations to associate a directed or undirected graph to $S$, and we can study the algebraic properties of $S$ in terms of properties of associated graphs.

Directed power graphs associated to semigroups were introduced by Kelarev and Quinn \cite{K}. If $S$ is a semigroup, then the directed power graph $\overrightarrow{\mathcal{P}(S)}$ of $S$ is a directed graph with $S$ as the set of all vertices and for any two distinct vertices $u$ and $v$ of $S$, there is an arc from $u$ to $v$ if $v=u^m$ for some positive integer $m$. Then Chakrabarty, Ghosh and Sen \cite{pet} defined the undirected power graph $\mathcal{P}(S)$ of a semigroup $S$ as the underlying undirected graph of $\overrightarrow{\mathcal{P}(S)}$. Thus two distinct elements $u$ and $v$ of $S$ are edge connected in $\mathcal{P}(S)$ if $u=v^{m}$ or $v=u^{m}$ for some positive integer $m$. They proved that for a finite group $G$, the undirected power graph $\mathcal{P}(G)$ is complete if and only if $G$ is a cyclic group of order $1$ or $p^{m}$ for some prime $p$ and positive integer $m$. In \cite{Cameran}, Cameron showed that for two finite commutative groups $G_1$ and $G_2$, $\mathcal{P}(G_1) \cong \mathcal{P}(G_2)$ implies that $G_1 \cong G_2$. They also observed that two finite groups with isomorphic undirected power graphs have the same number of elements of each order.

Singh and Manilal \cite{cb} introduced strong power graph as a generalization of the undirected power graph of a finite group. Let $G$ be a group of $n$ elements. The strong power graph $\mathcal{P}_s(G)$ of $G$ is a simple undirected graph  whose vertices consists of the elements of $G$ and two distinct vertices $a$ and $b$ are adjacent in $\mathcal{P}_s(G)$ if $a^{m_1}$=$b^{m_2}$ for some positive integers $m_1, m_2 < n$. Thus a finite group $G$ is noncyclic if and only if $\mathcal{P}_s(G)$ of is complete. Also $\mathcal{P}_s(G)$ is connected if and only if $n$ is composite.

Here we give several graph theoretic and spectral characterizations of the strong power graph of a finite group. The strong power graph $\mathcal{P}_s(G)$ is a Cayley graph of some group if and only if $G$ is noncyclic. We give a complete list of the finite groups $G$ such that $\mathcal{P}_s(G)$ is a line graph. Such characterizations of the strong power graphs are given in Section 2. In Section 3, the algebraic connectivity and the chromatic number of $\mathcal{P}_s(G)$ have been found.

For any graph $\Gamma$, let $A(\Gamma)$ be the adjacency  matrix and $D(\Gamma)$ be the diagonal of vertex degrees. Then the Laplacian matrix of $\Gamma$ is defined as $L(\Gamma)$=$D(\Gamma)-A(\Gamma)$. Clearly $L(\Gamma)$ is a real symmetric matrix and it is well known that $ L(\Gamma)$ is a positive semidefinite matrix with $0$ as the smallest eigen value. Thus we can assume that the Laplacian eigen values are $\lambda_1 \geq \lambda_2 \geq \lambda_3 \geq \cdots \geq \lambda_n = 0$. Among all Laplacian eigen values of a graph, one of the most popular is the second smallest, called by Fiedler \cite{f}, the algebraic connectivity of a graph. It is a good parameter to measure, how well a graph  is connected. For example, a graph is connected if and only if its algebraic connectivity is non-zero. According to Mohar \cite{m} the Laplacian eigen values are more intuitive and much more important than the eigen values of the adjacency matrix. We give a complete characterization of the Laplacian spectrum of strong power graph of any finite group in Section 4.

In Section 5, we have derived an explicit formula for the permanent of the Laplacian matrix of strong power graphs of any finite group. Let $A=(a_{ij})$ be a square matrix of order $n$, then the permanent of $A$ is denoted by per$(A)=\sum_{\sigma \in S_n}a_{1\sigma(1)}a_{2\sigma(2)}\cdots a_{n\sigma(n)}$, where $S_n$ is the set of all permutations of $1, 2, \cdots, n $. The permanent function was introduced by Binet and independently by Cauchy in 1812. We refer to \cite{H} and \cite{MH} for more on permanent.

\section{Representing strong power graph as line graph and Cayley graph}
In this section we characterize the groups $G$ such that the strong power graph $\mathcal{P}_s(G)$ is a line graph and Cayley graph. The line graph $L(\Gamma)$ of a graph $\Gamma$ is constructed by taking the edge of $\Gamma$ as vertices of $L(\Gamma)$, and joining two vertices in $L(\Gamma)$ whenever the corresponding edges in $\Gamma$ have a common vertex. A graph $\Gamma$ is called a line graph if it is a line graph of some graph. The induced subgraph $T$ of a graph $\Gamma$ is a subgraph whose vertex set $V(T) \subseteq V(\Gamma)$ and all the edges whose endpoints are contained in $V(T)$. The graph obtained by taking the union of graphs $\Gamma$ and $\Pi$ with disjoint vertex set is the disjoined union or sum, denoted by $\Gamma +\Pi$. The following result is a fundamental characterization of the simple line graphs. Proof of this result can be found in \cite{w}.
\begin{lemma}
A simple graph $G$ is the line graph of some simple graph if and only if $G$ does not have any of the nine graphs below as an induced subgraph.
\end{lemma}
\begin{center}
    \includegraphics[width=6in]{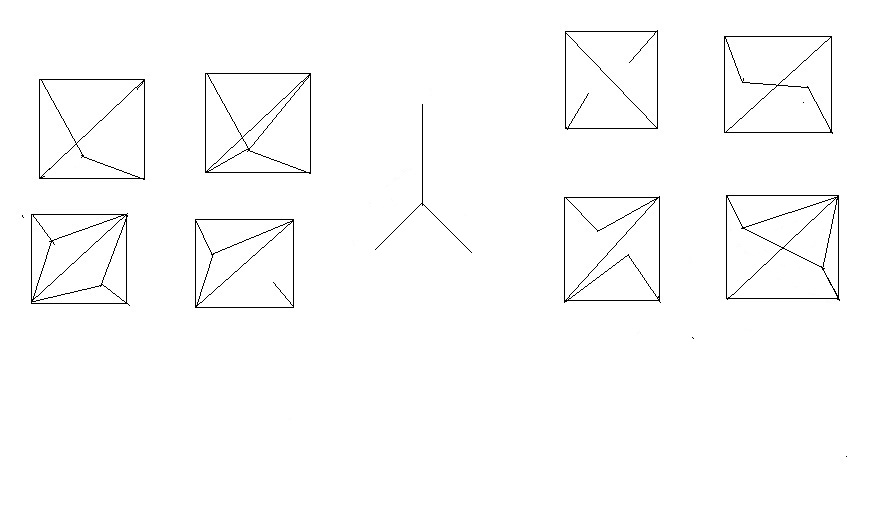}
\end{center}

For any positive integer $n$, $\mathbb{Z}_n$ denotes the group of all integers modulo $n$. Let $D'(n)=\{d \in \mathbb{N} \mid d|n, d \neq 1, n\}$. Thus $D'(n)=\emptyset$ for every prime $p$.  Then we have $\phi(n) > n-4 \Leftrightarrow n-\phi(n) < 4 \Leftrightarrow \sum_{d\in \acute D(n)}\phi(d)<3$  $\Leftrightarrow n=4, 9$ or a prime.
\begin{lemma}
If $\mathcal{P}_s(\mathbb{Z}_n)$ is a line graph then $n=4, 9$ or a prime.
\end{lemma}
\begin{proof}
If possible, on the contrary, suppose that $n\neq 4, 9$ and a prime. Then it follows that $\phi(n) \leq n-4$ which implies that $n-\phi(n)-1\geq 3$, and hence $\mathbb{Z}_n$ has at least three non-generators say, $b_1, b_2$  and $b_3$. In this case the following graph is an induced subgraph of $\mathcal{P}_s(\mathbb{Z}_n)$.
\begin{center}
    \includegraphics[width=4in]{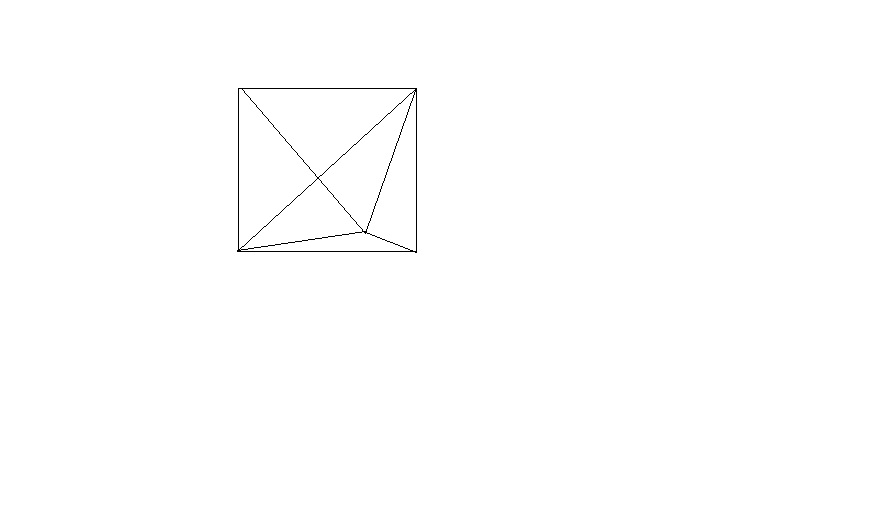}
\end{center}

Therefore, by Lemma 2.1, $\mathcal{P}_s(\mathbb{Z}_n)$ can not be a line graph.
\end{proof}

Now we are ready to characterize all groups $G$ such that $\mathcal{P}_s(G)$ is a line graph.
\begin{theorem}
Let $G$ be a finite group of order $n$.
\begin{enumerate}
\item If $G$ is non-cyclic, then $\mathcal{P}_s(G)$ is a line graph of $K_{1, n}$.
\item If $G$ is cyclic, then $\mathcal{P}_s(G)$ is line graph if and only if $n=4, 9$ or a prime.
\end{enumerate}
\end{theorem}
\begin{proof} 1. If $G$ is non-cyclic then $\mathcal{P}_s(G)$ is a complete graph with $n$ vertices, and hence $\mathcal{P}_s(G)=L(K_{1, n})$.
\\ 2. Let $G$ be a cyclic group. Since strong power graphs of two isomorphic groups are isomorphic, so it is sufficient to prove the result for $G=\mathbb{Z}_n$. First suppose that $n=4, 9$ or a prime. We have $\mathcal{P}_s(\mathbb{Z}_4)$ is the line graph of the graph
\begin{center}
\includegraphics[width=2in]{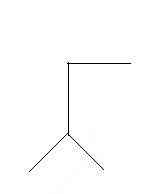}
\end{center}
 and $\mathcal{P}_s(\mathbb{Z}_9)$ is the line graph of the graph
 \begin{center}
  \includegraphics[width=4in]{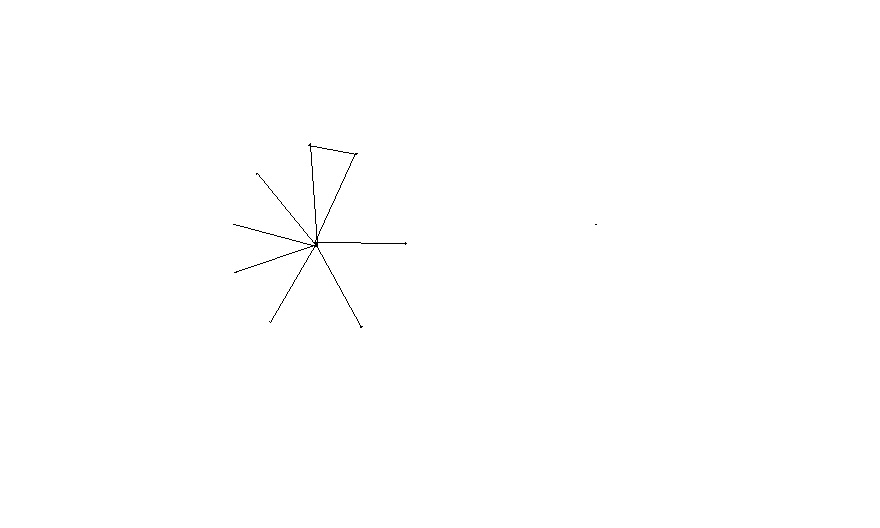}
  \end{center}

If $n=p$, a prime number then $\mathcal{P}_s(\mathbb{Z}_p)$ is the line graph of $K_{1, p-1}+K_2$.

converse, follows from the Lemma 2.2.
\end{proof}

The notion of Caley graph was first introduced by Sir Arthur Cayley (1878) to produce a large class of natural example of vertex transitive graphs. Let $G$ be a finite group of order $n$ and $S$ be a subset of $G$ not containing $e$, the identity of $G$ satisfying $S^{-1}=\{s^{-1}:s\in S\}=S$. Then the Cayley graph of $G$ is denoted by $C(G, S)$ is the graph whose vertex set is $G$ and two distinct vertices $g, h$ are edge connected if $gh^{-1}\in S$.
\begin{theorem}
Let $G$ be a group of order $n$. Then $\mathcal{P}_s(G)$ is Cayley graph of some group $G_1$ if and only if $G$ is noncyclic.
\end{theorem}
\begin{proof}
Let $G$ is a noncyclic group of order $n$. Then $\mathcal{P}_s(G)$ is a complete graph. Consider $G_1=G$ and $S=G\setminus \{e\}$, then clearly $\mathcal{P}_s(G)$  is Cayley graph of $G_1=G$.

If $G$ is a cyclic group of order $n$, then $\mathcal{P}_s(G)$ is not regular. But every Cayley graph is regular. Hence $\mathcal{P}_s(G)$ is not a Cayley graph of any group.
\end{proof}

\section{Vertex connectivity and chromatic number of  strong power graphs of  finite groups}
In this section we have discussed about vertex connectivity  and chromatic number of the strong power graph $\mathcal{P}_s(G)$ of any finite group G.

The {\bf vertex connectivity or connectivity} of a  graph $\Gamma $ is the minimum number of vertices in a vertex set $S$ such that $\Gamma$ $\setminus$ $S$ is disconnected or has only vertex. It is denoted by $ \kappa(\Gamma)$.Thus if $\Gamma$ is a complete graph with  $n$ vertices  then $\kappa(\Gamma)$=$n-1$, and if $\Gamma$ is a disconnected graph then $\kappa(\Gamma)$=$0$. The strong power graph $\mathcal{P}_s(G)$ of every finite noncyclic group $G$ is complete. Thus we have:

\begin{theorem}
For any non cyclic  finite group G  of order n the vertex connectivity  $\kappa(\mathcal{P}_s(G))$ is n-1.
\end{theorem}

\begin{theorem}
Let G be a  cyclic group of order $ n $ $>$1.
\begin{itemize}
\item[(a)]
If n is a prime number then  $\kappa(\mathcal{P}_s(G))$= $0$.
\item[(b)]
If n is a composite number then  $\kappa( \mathcal{P}_s(G))$ = $n-\phi(n)-1$.
\end{itemize}

\begin{proof}
Every non identity element of the group $G$is  of order  $n$,and no element of $G$  is adjacent with the identity  element $e$ of the group $G$.Thus $\mathcal{P}_s(G)$ is a disconnected graph with two components ${e}$ and $G$ $\setminus$ ${e}$, which implies that $\kappa(\mathcal{P}_s(G))$=$0$. If $n$ is composite  then $n-\phi(n)-1$ elements of the finite cyclic group $G$ are non identity and non generators. Each of these non identity non generators is adjacent to every other vertex, which implies that to make the graph $\mathcal{P}_s(G)$ disconnected we have to remove at least these $n-\phi(n)-1$ non identity non generators. Thus $\kappa(\mathcal{P}_s(G))$ $\geq$ $n-\phi(n)-1$. Since each generator is adjacent to every non identity element, so the removal of these $n-\phi(n)-1$ non identity non generators  makes the remaining graph disconnected with two components, one containing $e$ only and other containing all the generators. Thus $\kappa(\mathcal{P}_s(G))$=$n-\phi(n)-1$.
\end{proof}
\end{theorem}
The {\bf chromatic number }of a graph $\Gamma$ is  the minimum number of colours required to colour the vertices so that adjacent vertices get distinct colours.
\begin{theorem}
 Let G be a  cyclic  group of order n,then the chromatic number $\chi(\mathcal{P}_s(G))$  of the strong power graph $\mathcal{P}_s(G)$ is n-1.\
 \begin{proof}
 The subgraph $\mathcal{P}_s(G)$ $\setminus$ $e$ is complete.
 Since the chromatic number of a complete  graph with $n-1$ vertices is $n-1$, so the chromatic number $\chi(\mathcal{P}_s(G))$ $\geq n-1$. Since the generators  of the cyclic group $G$ are not adjacent  with the identity  $e$ of $G$, either of of colours of the generators can be given to $e$ and hence $\chi(\mathcal{P}_s(G))$ $\leq n-1$. Thus $\chi(\mathcal{P}_s(G))$ =$n-1$.
 \end{proof}
\end{theorem}
\begin{result}
Let G be a non-cyclic group of order n, then the chromatic number of the strong power graph $\mathcal{P}_s(G)$ is n.\
\begin{proof}
Since the strong power graph $\mathcal{P}_s(G)$ of the cyclic group $G$ is complete, hence the result follows.
\end{proof}
\end{result}

\section{Laplacian spectrum of the strong power graphs of  finite groups}
\begin{theorem}
For each positive integer $n\geq2,$ let $\bar s_i$$(i=1,2,\cdots m) $ are non generators of
$\mathbb{Z}_n$ then $\Theta(\mathcal{P}_s(\mathbb{Z}_n)),x)$
=$x(x-n)^{n-\phi(n)-1}(x-n+\phi(n)+1)(x-n+1)^{\phi(n)-1}$, where m=$n-\phi(n)-1$.
\begin{proof}
The Laplacian matrix $L(\mathcal{P}_s(\mathbb{Z}_n))$ is the  n times n  matrix whose rows and
columns are indexed in order by the non generators $\bar s_i$$(i=1,2,\cdots m) $ and the
generators of $\mathbb{Z}_n$ and $\bar0$ is in last position. Then
$L(\mathcal{P}_s(\mathbb{Z}_n)) =
\left(
  \begin{array}{cccccccccc}
    n-1 & -1  & \cdots & -1 & -1  & \cdots & -1 & -1 \\
    -1 & n-1  & \cdots & -1 & -1 & \cdots & -1 & -1  \\

    \vdots & \vdots  & \ddots& \vdots & \vdots & \ddots & \vdots & \vdots\\
    -1 & -1  & \cdots & n-1 & -1  & \cdots & -1 & -1 \\
    -1 & -1  & \cdots & -1 & n-2  & \cdots & -1 & 0 \\

    \vdots & \vdots  & \ddots & \vdots & \vdots  & \ddots & \vdots & \vdots \\
    -1 & -1  & \cdots & -1 & -1  & \cdots & n-2 & 0 \\
    -1 & -1  & \cdots & -1 &  0  & \cdots & 0 & n-\phi(n)-1  \\
  \end{array}
\right)$
 \\Each row and column sum of the above matrix is zero. Then the characteristic polynomial of  $L(\mathcal{P}_s(\mathbb{Z}_n))$ is
$\Theta(L(\mathcal{P}_s(\mathbb{Z}_n)),x)$=

$\left|
  \begin{array}{cccccccccc}
    x-(n-1)    & \cdots & 1 & 1 & \cdots &1 & 1 \\

    \vdots  & \ddots & \vdots & \vdots &\ddots &1  & \vdots \\
    1   & \cdots & x-(n-1) & 1  & \cdots &1 &1  \\

    1   & \cdots & 1 & x-(n-2) & \cdots &1  & 0 \\

    \vdots   & \ddots &  \vdots & \vdots & \ddots &1  & \vdots \\
    1   & \cdots & 1 & 1  & \cdots  &x-(n-2) & 0 \\
    1   & \cdots & 1 & 0   & \cdots  &0& x-(n-\phi(n)-1) \\
   \end{array}
\right|$

Multiplying the first row of $\Theta(L(\mathcal{P}_s(\mathbb{Z}_n)),x)$ by $ (x-1)$
and applying the row operation $R^{'}_1 $=$R_1-R_2-R_3- \cdots -R_{(n-1)}-R_n$. Then the above
determinant becomes
\begin{align*}
\Theta(L(\mathcal{P}_s(\mathbb{Z}_n)),x)=
\frac{x(x-n)}{(x-1)}\Theta(L\bar{s_1}(\mathcal{P}_s(\mathbb{Z}_n)),x).
\end{align*}
Again multiplying the first row of $\Theta(L\bar{s_1}(\mathcal{P}_s(\mathbb{Z}_n)),x)$
by $(x-2)$ and applying the row operation $R^{'}_1 $=$R_1-R_2-R_3- \cdots -R_{(n-1)}$
we get
\begin{align*}
\Theta(L\bar{s_1}(\mathcal{P}_s(\mathbb{Z}_n)),x)=
\frac{(x-1)(x-n)}{(x-2)}\Theta(L\bar{s}_1,\overline{s}_2(\mathcal{P}_s(\mathbb{Z}_n)),x),
\end{align*}
and so
\begin{align*}
\Theta(L(\mathcal{P}_s(\mathbb{Z}_n)),x)=
\frac{x(x-n)^{2}}{(x-2)}\Theta(L\bar{s_1},\bar{s}_2(\mathcal{P}_s(\mathbb{Z}_n)),x).
\end{align*}
Continuing in this way we get
\begin{align*}
\Theta(L(\mathcal{P}_s(\mathbb{Z}_n)),x)=
\frac{x(x-n)^{(n-\phi(n)-1)}}{(x-n+\phi(n)+1)}\Theta(L\bar{s}_1,\overline{s}_2 \cdots \bar{s}_m
(\mathcal{P}_s(\mathbb{Z}_n)),x),
\end{align*}
 where

$\Theta(L\bar{s}_1,\overline{s}_2 \cdots \bar{s}_m (\mathcal{P}_s(\mathbb{Z}_n)),x)$=
$\left|
   \begin{array}{ccccc}
     x-\lambda & 1 & 1 & \cdots & 0 \\
     1 & x-\lambda & 1 & \cdots & 0 \\
     \vdots & \vdots & \vdots & \ddots & \vdots \\
     1 & 1 & 1 & \cdots & 0 \\
     0 & 0 & 0 & \cdots & x-m \\
   \end{array}
 \right|$

  $\lambda=n-2$, $m=n-\phi(n)-1$  and order of the determinant is $\phi(n)+1$. Now expanding $\Theta(L\bar{s}_1,\overline{s}_2 \cdots \bar{s}_m
  \mathcal{P}_s(\mathbb{Z}_n),x))$ with respect to last row we get

\begin{align*}
\Theta(L\bar{s}_1,\overline{s}_2 \cdots \bar{s}_m (\mathcal{P}_s(\mathbb{Z}_n)),x))=&
(-1)^{2\phi(n)}(x-m)\left|
                        \begin{array}{cccc}
                          x-\lambda & 1 & \cdots & 1 \\
                          1 & x- \lambda & \cdots & 1 \\
                          \vdots & \vdots & \ddots & \vdots \\
                          1 & 1 & \cdots & x-\lambda \\
                        \end{array}
                      \right|
\\ &=(x-n+\phi(n)+1)(x-\lambda+\phi(n)+1-2)(x-\lambda-1)^{(\phi(n)-1)}
\\ &=(x-n+\phi(n)+1)^{2}(x-n+2-1)^{(\phi(n)-1)}
\\&=(x-n+\phi(n)+1)^{2}(x-n+1)^{(\phi(n)-1)}.
\end{align*}
Thus $ \Theta(L(\mathcal{P}_s(\mathbb{Z}_n)),x)=
x(x-1)^{n-\phi(n)-1}(x-n+\phi(n)+1)(x-n+1)^{\phi(n)-1}$.

\end{proof}

\end{theorem}

Let G is a cyclic group of order n, then $G$ is isomorphic to $\mathbb{Z}_n$; and the strong
power graphs $\mathcal{P}_s(G)$ and $\mathcal{P}_s(\mathbb{Z}_n)$ of $G$ and $\mathbb{Z}_n$
respectively are isomorphic. Hence the graphs $\mathcal{P}_s(G)$ and
$\mathcal{P}_s(\mathbb{Z}_n)$ have the same Laplacian spectrum. So by theorem $2.1$ we have :
 \begin{proposition}
 If $G$ is a cyclic group of order $n$,then the laplacian spectrum of  $\mathcal{P}_s(G)$ is
\begin{align*}
 \left(
  \begin{array}{cccc}
    0 & n & n-\phi(n)-1  & n-1 \\
    1 & n-\phi(n)-1 & 1 & \phi(n)-1 \\
  \end{array}
\right)
\end{align*}
\end{proposition}

For any  non cyclic group $G$, the strong power graph $\mathcal{P}_s(G)$ is complete \cite{Chattopa}. So their
Laplacian spectrum is given by :
\begin{proposition}
Let $G$ be a noncyclic group of order $n$, then the Laplacian spectrum of  $\mathcal{P}_s(G)$
is
  \begin{align*}
  \left(
  \begin{array}{cc}
   0 & n \\
   1 & n-1 \\
   \end{array}
   \right)
   \end{align*}
\end{proposition}
The algebraic connectivity of a graph $\Gamma$, denoted by $a(\Gamma)$, is the second smallest Laplcian eigen value of $\Gamma$ \cite{f}. Now, the algebraic connectivity has received special attention due to its huge applications on connectivity problems, isoperimetric numbers, genus, combinatorial optimizations and many other problems.    Thus by Theorem $2.1$ and proposition $2.3$ we have:

\begin{corollary}
Let $G$ be a group of order $n$.
\begin{enumerate}
\item[1] If $G$ is a cyclic group then $a(\mathcal{P}_s(G))=n-\phi(n)-1$.
\item[2]If $G$ is a noncyclic group then  $ a(\mathcal{P}_s(G))=n $.
\end{enumerate}
\end{corollary}

Another important application of Laplacian spectrum is on the number of spanning trees of a graph. A spanning tree $T$ of a graph $\Gamma$ is a subgraph which is a tree having same vertex set is same as   $\Gamma$. If $\lambda_1$
$\geq$$\lambda_2$$\geq$$\lambda_3$$\geq$$\cdots$$\geq$$\lambda_n$=$0$  are the Laplacian eigenvalues of a graph $\Gamma$ of $n$-vertices, then the number of spanning trees of $\Gamma$ is denoted by $\tau(\Gamma)$ is given by $\frac{\lambda_1 \lambda_2 \cdots \lambda_{n-1}}{n}$ [Theorem 4.11; \cite{Bapat}]. Thus from the Proposition $2.2$ and Proposition $2.3$ we have:
\begin{corollary}
Let $G$ be a group of order $n$.
\begin{enumerate}
\item[1] If $G$ is a cyclic group then $\tau(\mathcal{P}_s(G))=n^{n-\phi(n)-2}(n-\phi(n)-1)(n-1)^{\phi(n)-1}$.
\item[2] If $G$ is a noncyclic group then $\tau(\mathcal{P}_s(G))=n^{n-2}$.
\end{enumerate}

\end{corollary}

The graph energy is defined in terms of the spectrum of the adjacency matrix. Depending on the  well-developed spectral theory of the Laplacian matrix, recently Gutman et. al \cite {GZ} have defined the Laplacian energy of a graph $\Gamma$ with $n$ vertices and $m$ edges as: $LE(\Gamma)=\sum_{i=1}^{n}|\lambda_i-\frac{2m}{n}|$, where $\lambda_1$
$\geq$$\lambda_2$$\geq$$\lambda_3$$\geq$$\cdots$$\geq$$\lambda_n$=$0$ are the Laplacian eigen values of the graph $\Gamma$. This definition has been adjusted so that the Laplcian energy becomes equal to the energy for any regular graph. For various properties of Laplacian energy we refer \cite{GZF}, \cite{HGW}, \cite{Y}. From Proposition $2.2$ and Proposition $2.3$ we have
\begin{corollary}
Let $G$ be a finite group of order $n$.
\begin{enumerate}
\item[1]If $G$ is cyclic then $LE(\mathcal{P}_s(G))=2(n-1)-\frac{4\phi(n)}{n}$.
\item[2]If $G$ is noncyclic then $LE(\mathcal{P}_s(G))=2(n-1)$.
\end{enumerate}
\end{corollary}

\section{Permanent of the Laplacian of strong power graph }

Let us recall the definition of permanent of a square matrix. For any square matrix $A=(a_{ij}) $  of  order $n$, the  permanent  of A is denoted by per(A) and per(A) =$\sum_{\sigma \in S_n}a_{1\sigma(1)}a_{2\sigma(2)}\cdots a_{n\sigma(n)}$. It is quite difficult to determine the permanent of a square matrix. In this section we have determined the permanent of the Laplacian matrix of strong power graph of any finite group explicitly. Our method is based on the following observation. Let A =$(a_{ij})$ be a matrix of order $n$, then per(A) is equal to the coefficient  of $x_1x_2\cdots x_n$ in the expression $(a_{11}x_1+a_{12}x_2+\cdots a_{1n}x_1n)(a_{21}x_1+a_{22}x_2+\cdots a_{2n}x_n) \cdots (a_{n1}x_1+a_{n2}x_2+ \cdots a_{nn}x_n)$. Throughout the rest of this section we make the following convention: for any $n$-functions $f_1(x)f_2(x)\cdots \widehat{f_i(x)}\cdots f_n(x)=f_1(x)f_2(x) \cdots f_{i-1}(x)f_{i+1}(x)\cdots f_n(x) $ and for any $n$ variables $x_1, x_2, \cdots x_n $, the coefficient of $x_{n_1}x_{n_2} \cdots x_{n_k}$ in a polynomial $F(x_1, x_2, \cdots x_n)$ will be denoted by $C_{x_{n_1}x_{n_2}\cdots x_{n_k}} ( F(x_1, x_2, \cdots x_n))$. Here we first find the permanent of adjacency matrix of any graph with $m+n+1$ vertices such that $m+n$ vertices form a clique and the rest vertex is adjacent with $n$ vertices. We have the following Lemma,
\begin{lemma}
The permanent of adjacency matrix of any graph with $m+n+1$ vertices such that $m+n$ vertices form a clique and the rest vertex is adjacent with $n$ vertices is \\
$$n[\sum_{r=1}^{m+n}(-1)^{r-1}(m+n-r)!\{\left(
                                      \begin{array}{c}
                                        m+n-1 \\
                                        r-1 \\
                                      \end{array}
                                    \right)
  +(n-1)\left(
          \begin{array}{c}
            m+n-2 \\
            r-1 \\
          \end{array}
        \right) \}].$$
\end{lemma}
\begin{proof}
The required permanent is the coefficient of $x_1x_2 \cdots x_{m+n+1}$ in
$F(x_1,x_2,\cdots x_{m+n+1})=(X-x_1)(X-x_2)\cdots(X-x_m)(X-x_{m+1}+x_{m+n+1})(X-x_{m+2}+x_{m+n+1})$ $\cdots(X-x_{m+n}+x_{m+n+1})(x_{m+1}+x_{m+2}+\cdots x_{m+n})$, where $X=x_1+x_2+\cdots x_{m+n}$.

Now we have  $ F(x_1, x_2, \cdots x_{m+n+1})=\prod_{i=1}^{m}(X-x_i) [ x_{m+n+1}^{n} +\cdots+ x_{m+n+1}\sum_{i=1}^{n}(X-x_{m+1})(X-x_{m+2})\cdots\widehat{(X-x_{m+i})}\cdots(X-x_{m+n}) +\prod_{i=1}^{n}(X-x_{m+i})](\sum_{i=1}^{n}x_{m+i})$ which shows that

$C_{x_1x_2 \cdots x_{m+n+1}}(F(x_1, x_2, \cdots x_{m+n+1}))$
\begin{align*}
&=C_{x_1x_2 \cdots x_{m+n}} \prod_{i=1}^{m}(X-x_i)[\sum_{i=1}^{n}(X-x_{m+1})(X-x_{m+2})\cdots\widehat{(X-x_{m+i})} \cdots(X-x_{m+n})](\sum_{i=1}^{n}x_{m+i}).\\
&=n C_{x_1x_2 \cdots x_{m+n}}(x_{m+1}\prod_{i=1}^{m}(X-x_i)[\sum_{i=1}^{n}(X-x_{m+1})(X-x_{m+2})\cdots\widehat{(X-x_{m+i})} \cdots(X-x_{m+n})])\\
&=n C_{x_1x_2 \cdots x_m x_{m+2} \cdots x_{m+n}}(\prod_{i=1}^{m}(X-x_i)[\sum_{i=1}^{n}(X-x_{m+1})(X-x_{m+2})\cdots\widehat{(X-x_{m+i})} \cdots(X-x_{m+n})])\\
&=n[C_{x_1x_2 \cdots x_m x_{m+2} \cdots x_{m+n}}((X-x_1)\cdots (X-x_m)(X-x_{m+2})\cdots (X-x_{m+n}))+ \\
& \hspace{1.5cm} C_{x_1x_2 \cdots x_m x_{m+2} \cdots x_{m+n}}((X-x_1)\cdots (X-x_m)(X-x_{m+1})(X-x_{m+3})\cdots (X-x_{m+n}))+\cdots \\  \cdots
& \hspace{3.5cm} +C_{x_1x_2 \cdots x_m x_{m+2} \cdots x_{m+n}}((X-x_1)\cdots (X-x_m)(X-x_{m+1})\cdots (X-x_{m+n-1}))]
\end{align*}

Now, $(X-x_1)(X-x_2)\cdots(X-x_m)(X-x_{m+2})\cdots(X-x_{m+n})= X^{m+n-1}-X^{m+n-2} \sum_{i\neq(m+1)}x_i+\cdots+(-1)^{m+n-1}x_1x_2 \cdots x_mx_{m+2}\cdots x_{m+n}$,
shows that $C_{x_1x_2 \cdots x_m x_{m+2} \cdots x_{m+n}}((X-x_1)(X-x_2)\cdots(X-x_m)(X-x_{m+2})\cdots(X-x_{m+n})) = \sum_{r=1}^{m+n}(-1)^{r-1}(m+n-r)!\left(
                                     \begin{array}{c}
                                       m+n-1 \\
                                       r-1 \\
                                     \end{array}
                                   \right)$,
and $(X-x_1)(X-x_2)\cdots(X-x_m)(X-x_{m+1})(X-x_{m+3})\cdots(X-x_{m+n})$=
$X^{m+n-1}-X^{m+n-2}\sum_{i\neq(m+2)}x_i+\cdots+(-1)^{m+n-1}x_1x_2\cdots x_mx_{m+1}x_{m+3} \cdots x_{m+n}$ shows that $C_{x_1x_2 \cdots x_m x_{m+2} \cdots x_{m+n}}((X-x_1)(X-x_2)\cdots(X-x_m)(X-x_{m+1})(X-x_{m+3})\cdots(X-x_{m+n}))=
\sum_{r=1}^{m+n}(-1)^{r-1}(m+n-r)!\left(
                                     \begin{array}{c}
                                       m+n-2 \\
                                       r-1 \\
                                     \end{array}
                                   \right)$.

Also $C_{x_1x_2 \cdots x_m x_{m+2} \cdots x_{m+n}}((X-x_1)(X-x_2)\cdots(X-x_{m+2})(X-x_{m+4})\cdots(X-x_{m+n}))\\=C_{x_1x_2 \cdots x_m x_{m+2} \cdots x_{m+n}}((X-x_1)(X-x_2)\cdots(X-x_{m+3})(X-x_{m+5})\cdots(X-x_{m+n}))=\cdots = C_{x_1x_2 \cdots x_m x_{m+2} \cdots x_{m+n}}((X-x_1)(X-x_2)\cdots(X-x_m)\cdots(X-x_{m+n-1}))$\\=$\sum_{r=1}^{m+n}(-1)^{r-1}(m+n-r)!\left(
                                  \begin{array}{c}
                                       m+n-2 \\
                                       r-1 \\
                                     \end{array}
                                   \right)$.

Hence the permanent of the adjacency matrix of the stated graph is
$$n[\sum_{r=1}^{m+n}(-1)^{r-1}(m+n-r)!\{\left(
                                      \begin{array}{c}
                                        m+n-1 \\
                                        r-1 \\
                                      \end{array}
                                    \right)+(n-1)\left(
                                                   \begin{array}{c}
                                                     m+n-2 \\
                                                     r-1 \\
                                                   \end{array}
                                                 \right)\}].$$
\end{proof}

Now from the above lemma we get the permanent of adjacency matrix of strong power graph of any finite cyclic group.
Let $G$ be a cyclic group of order $n$. Then $G$ has $m=\phi(n)$ generators none of which is adjacent to the identity element $e$ of $G$. Thus the set $G \setminus {e}$ of all $n-1$ non-identity vertices forms a clique and identity $e$ is adjacent to each of the $n-(\phi(n)+1)$ non-identity non-generators. Hence from Lemma $3.1$ it follows immediately that:
\begin{theorem}
Let $G$ be a cyclic  group of order $n$, then the permanent of adjacency matrix of strong power graph of $G$ is  \\
$$(n-\phi(n)-1)[\sum_{r=1}^{n-1}(-1)^{r-1}(n-1-r)!\{\left(
                                      \begin{array}{c}
                                       n-2 \\
                                        r-1 \\
                                      \end{array}
                                    \right)+(n-2-\phi(n))\left(
                                                   \begin{array}{c}
                                                     n-3 \\
                                                     r-1 \\
                                                   \end{array}
                                                 \right)\}
].$$
\end{theorem}

Now we compute the permanent of the Laplacian matrix of a graph and hence compute the permanent of laplacian matrix of strong power graph of any finite group.
\begin{lemma}
The permanent of the Laplcian matrix of a graph $\Gamma$ with $m+n+1$ vertices such that $m+n$ vertices form a clique and the rest vertex is edge connected with $n$ vertices is $\sum_{r=1}^{m+n}(m+n-r)!F_r(d)$, where $F_r(d)=\sum_{i+j=r-1}\left(
                                                                                                    \begin{array}{c}
                                                                                                      m \\
                                                                                                      i \\
                                                                                                    \end{array}
                                                                                                  \right)
(d+2)^{j}(d+1)^{i}[n\left(
                      \begin{array}{c}
                        n-1 \\
                        j \\
                      \end{array}
                    \right)
+n(n-1)\left(
         \begin{array}{c}
           n-2 \\
           j \\
         \end{array}
       \right)
+(-1)^{m+n-r+1}(d-m+1)(m+n-r+1)\left(
                                 \begin{array}{c}
                                   n \\
                                   j \\
                                 \end{array}
                               \right)
]$, and $d=m+n-1$.
\end{lemma}
\begin{proof}
 Consider
$F(x_1,x_2,\cdots,x_{m+n+1})=\prod_{i=1}^{m}(X+(d+1)x_i)\prod_{i=1}^{n}(X+(d+2)x_{m+i}-x_{m+n+1})((d-m+1)x_{m+n+1}+\sum_{i=1}^{n}-x_{m+i})$,
 where $X=-(x_1+x_2+\cdots +x_{m+n})$.
Then the permanent of the Laplacian matrix of $\Gamma$ is  $C_{x_1x_2 \cdots x_{m+n+1}}(F(x_1,x_2,\cdots,x_{m+n+1}))$  which is equal to
   $C_{x_1x_2\cdots x_{m+n+1}}(\prod_{i=1}^{m}(X+(d+1)x_i)\prod_{i=1}^{n}(X+(d+2)x_{m+i}-x_{m+n+1})(\sum_{i=1}^{n}-x_{m+i}))+ C_{x_1x_2\cdots x_{m+n}}((d-m+1)\prod_{i=1}^{m}X+(d+1)x_i)\prod_{i=1}^{n}(X+(d+2)x_{m+i}))$.
 Now  $(d-m+1)\prod_{i=1}^{m}(X+(d+1)x_i)\prod_{i=1}^{n}(X+(d+2)x_{m+i})=(d-m+1)[X^{m+n}+X^{m+n-1}\sum f_1(d)x_i+X^{m+n-2}\sum f_{12}(d)x_1x_2 + \cdots +f_{123\cdots m+n}(d)x_1x_2 \cdots x_{m+n}]$, where $f_{123 \cdots j}(d)$ is  a product of some $(d+1)$ and $(d+2)$ which is clear from the context.  So  $C_{x_1x_2\cdots x_{m+n}}  ((d-m+1)\prod_{i=1}^{m}(X+(d+1)x_i)\prod_{i=1}^{n}(X+(d+2)x_{m+i})) = (d-m+1)\sum_{r=1}^{m+n+1}(-1)^{m+n-r+1}(m+n-r+1)!f_r(d)$, where $f_r(d)=\sum_{i+j=r-1}\left(
                                                                                       \begin{array}{c}
                                                                                         m \\
                                                                                         i \\
                                                                                       \end{array}
                                                                                     \right)\left(
                                                                                              \begin{array}{c}
                                                                                                n \\
                                                                                                j \\
                                                                                              \end{array}
                                                                                            \right)(d+2)^{j}(d+1)^{i}.$
Now proceeding in the proof of Lemma $3.1$ we get $C_{x_1x_2\cdots x_{m+n+1}}  (\prod_{i=1}^{m}(X+(d+1)x_i)\prod_{i=1}^{n}(X+(d+2)x_{m+i}-x_{m+n+1})(\sum_{i=1}^{n}-x_{m+i}) =n\sum_{r=1}^{m+n}(-1)^{m+n-r}(m+n-r)![\sum_{i+j=r-1}\left(
                                                       \begin{array}{c}
                                                         m \\
                                                         i \\
                                                       \end{array}
                                                     \right)(\left(
                                                               \begin{array}{c}
                                                                 n-1 \\
                                                                 j \\
                                                               \end{array}
                                                             \right)+(n-1)\left(
                                                                            \begin{array}{c}
                                                                              n-2 \\
                                                                              j \\
                                                                            \end{array}
                                                                          \right)
                                                     )(d+2)^{j}(d+1)^{i}]$.

Hence the  permanent of the Laplacian matrix of $\Gamma$ is
$$\sum_{r=1}^{m+n}\{(-1)^{m+n-r}(m+n-r)!F_r(d)\}+(d-m+1)\sum_{i+j=m+n}\left(
                                                                                                       \begin{array}{c}
                                                                                                         m \\
                                                                                                         i \\
                                                                                                       \end{array}
                                                                                                     \right)\left(
                                                                                                              \begin{array}{c}
                                                                                                                n \\
                                                                                                                j \\
                                                                                                              \end{array}
                                                                                                            \right)(d+2)^{j}(d+1)^{i},$$
where $F_r(d)=\sum_{i+j=r-1}\left(
                                  \begin{array}{c}
                                    m \\
                                    i \\
                                  \end{array}
                                \right)(d+2)^{j}(d+1)^{i}\{n\left(
                                                             \begin{array}{c}
                                                               n-1 \\
                                                               j \\
                                                             \end{array}
                                                           \right)+n(n-1)\left(
                                                                           \begin{array}{c}
                                                                             n-2 \\
                                                                             j \\
                                                                           \end{array}
                                                                         \right)-(d-m+1)(m+n-r+1)\left(
                                                                                                   \begin{array}{c}
                                                                                                     n \\
                                                                                                     j \\
                                                                                                   \end{array}
                                                                                                 \right)\}$.

\end{proof}

The strong power graph of a cyclic group of order $n$ is a graph of same type as in Lemma $3.3$. So we have:

\begin{theorem}
Let $G$ be a cyclic group of order $n$, then the permanent of Laplacian matrix of strong power graph of $G$ is \\ $\sum_{r=1}^{n-1}\{(-1)^{n-r-1})(n-r-1)!F_r(d)\}+(n-\phi(n)-1)\sum_{i+j=n-1}\left(
                                                                              \begin{array}{c}
                                                                                \phi(n) \\
                                                                                i \\
                                                                              \end{array}
                                                                            \right)\left(
                                                                                     \begin{array}{c}
                                                                                       n-\phi(n)-1 \\
                                                                                       j \\
                                                                                     \end{array}
                                                                                   \right)n^{j}(n-1)^{i}$,
where $F_r(d)=\sum_{i+j=r-1}\left(
                              \begin{array}{c}
                                \phi(n) \\
                                i \\
                              \end{array}
                            \right)n^{j}(n-1)^{i}\{(n-\phi(n)-1)\left(
                                                                  \begin{array}{c}
                                                                    n-\phi(n)-2\\
                                                                    j \\
                                                                  \end{array}
                                                                \right)+(n-\phi(n)-1)(n-\phi(n)-2)\left(
                                                                                                    \begin{array}{c}
                                                                                                      n-\phi(n)-3 \\
                                                                                                      j \\
                                                                                                    \end{array}
                                                                                                  \right)
-(n-\phi(n)-1)(n-r)\left(
                     \begin{array}{c}
                       n-\phi(n)-1 \\
                       j \\
                     \end{array}
                   \right)
                            \}$.
\end{theorem}

 For any noncyclic group $G$ of order $n$ the strong power graph $\mathcal{P}_s(G)$ is complete, so we have:
\begin{theorem}
The permanent of the Laplacian matrix of the strong power graph $\mathcal{P}_s(G)$ of a noncyclic group $G$ of order $n$ is
$(-1)^{n}n!(1-\frac{n}{1!}+\frac{n^{2}}{2!}-\frac{n^{3}}{3!}+\cdots + (-1)^{n}\frac{n^{n}}{n!})$.
\end{theorem}

\bibliographystyle{amsplain}

\end{document}